\documentclass[12pt]{amsart}

\usepackage{amsmath, amstext, amsthm, amssymb, amsfonts, latexsym}

\numberwithin{equation}{section}

\newcommand{\T}{\mathbb{T}}
\newcommand{\R}{\mathbb{R}}
\newcommand{\Z}{\mathbb{Z}}
\newcommand{\TT}{\mathbb{T}^2}
\newcommand{\RR}{\mathbb{R}^2}
\newcommand{\ZZ}{\mathbb{Z}^2}

\newcommand{\cU}{\mathcal{U}}

\newcommand{\tU}{\tilde{U}}
\newcommand{\tV}{\tilde{V}}

\newcommand{\tf}{\tilde{f}}
\newcommand{\tp}{\tilde{p}}
\newcommand{\tq}{\tilde{q}}
\newcommand{\tx}{\tilde{x}}
\newcommand{\ty}{\tilde{y}}
\newcommand{\tz}{\tilde{z}}
\newcommand{\tgamma}{\tilde{\gamma}}
\newcommand{\tsigma}{\tilde{\sigma}}
\newcommand{\tS}{\tilde{\Sigma}}
\newcommand{\tG}{\tilde{\Gamma}}

\DeclareMathOperator{\image}{Im}

\newtheorem{theorem}{Theorem}[section]

\newtheorem*{theorem2}{Theorem 2.2 restated}
\newtheorem{proposition}[theorem]{Proposition}
\newtheorem{lemma}[theorem]{Lemma}

\newtheorem{definition}[theorem]{Definition}

\title{Transitivity of conservative toral endomorphisms}
\date{\today}
\author[Martin Andersson]{Martin Andersson}

\email{martin@id.uff.br}
\begin{thanks}{This work was carried out with the support from 973 project 2011CB808002 (China) and CNPq (Brazil).}
\end{thanks}

\begin{document}

\begin{abstract}
It is shown that if a non-invertible area preserving local homeomorphism on $\TT$ is homotopic to a linear expanding or hyperbolic endomorphism, then it must be topologically transitive. This gives a complete characterization, in any smoothness category, of those homotopy classes of conservative endomorphisms that consist entirely of transitive maps.
\end{abstract}

\maketitle

\section{Introduction}

Robust transitivity and stable ergodicity are two central themes in dynamical systems. They run parallel, feeding from each others, and are often seen as analogous concepts in the dissipative versus conservative setting. Although typically studied in the context of diffeomorphisms, interest in robustly transitive endomorphisms is growing. Sumi \cite{MR1469436} gave the first example of non-hyperbolic robustly transitive local diffeomorphisms of $\TT$. Later, Gan and He \cite{MR3043026} proved that such examples exist in every homotopy class of toral endomorphisms.  Lizana and Pujals \cite{MR3082540} have recently given conditions under which endomorphisms of $\mathbb{T}^n$ are robustly transitive, and the interplay of robust transitivity with periodic points and invariant measures is further studied in \cite{LPV}. On the other hand, stable ergodicity of endomorphisms has never been seriously studied. This is somewhat surprising, since many questions about the dynamics of diffeomorphisms with some hyperbolicity has a simpler, and sometimes more revealing, lower dimensional non-invertible analogue. In fact, the present work originated from the question of whether every volume preserving smooth enough diffeomorphism on $\mathbb{T}^3$ which has dominated splitting and is isotopic to Anosov is ergodic, or at least transitive. It was suggested by Enrique Pujals that the question may be elucidated by considering local diffeomorphisms homotopic to an expanding map on $\TT$. This approach turned out to be very fruitful because it reveals more clearly how topological information given by the action in the fundamental group translates into dynamics. Although obtaining ergodicity in such a gneral context may currently be without reach, the present work answers the question of transitivity in the affirmative. The analogous case of diffeomorphisms isotopic to Anosov on $\T^3$ is treated separately in \cite{2015arXiv150306501A}.

As it turns out, transitivity of conservative endomorphisms homotopic to expanding or hyperbolic linear maps is a purely topological phenomenon, resulting only from the way the map acts in the fundamental group, together with the constraint of being area preserving and non-invertible. The whole argument relies on a surprisingly simple description of the topology of regular open invariant sets. In particular, it is independent of any infinitesimal analysis such as Lyapunov exponents or  domination. This is what allows us to work within the setting of local homeomorphisms. The low regularity comes "for free"; it is not pursued for its own sake.

An obvious question is of course whether there is a non-conservative analogue of Theorem \ref{main}. A natural candidate would be to ask that the endomorphism in question be area expanding, i.e. that it has a Jacobian larger than one at every point (see \cite[Theorem 2]{MR3082540}). But the condition of being area preserving cannot be replaced with forward area expansion, at least not in those homotopy classes in which the linear representative has integer eigenvalues. For example, let $f_1 : \R/\Z \to \R / \Z$ be the linear expanding map $x \mapsto 2x \mod 1$, and let $f_2 : \R / \Z \to \R / \Z $ be a deformation of $f_1$ having derivative larger than $\frac{1}{2}$ at every point, but  mapping a neighborhood of the origin into itself. Then the map $f: \TT \to \TT$ defined by $f(x,y) = (f_1(x) , f_2(y))$ is not transitive, even though it is  homotopic to a linear expanding map (with $2$ as a double eigenvalue) and has a Jacobian strictly larger than one at every point. I thank Radu Saghin for pointing this out. I would also like to thank Shaobo Gan, Enrique Pujals, and Rusong Zheng for helpful discussions; the anonymous referee for important comments and a simple proof of Lemma \ref{invariant under iterate}; and the hospitality of Peking University where this work took place.

\section{The result}

Let $\TT = \RR / \ZZ$ be the two dimensional torus endowed with its Haar measure $\lambda$. By a toral endomorphism we mean a local homeomorphism $f: \TT \to \TT$. In other words, an endomorphism is a covering map from the torus to itself. We prefer the term endomorphism to covering map because we consider it as a dynamical system rather than a tool of algebraic topology. We say that  $f: \TT \to \TT$ is area preserving (or conservative) if $\lambda(f^{-1}(A)) = \lambda(A)$ for every Borel measurable set $A \subset \TT$. When $f$ is of class $C^1$, this is equivalent to say that $\sum_{x' \in f^{-1}(x)} |\det Df(x')|^{-1} = 1$ for every $x \in \TT$.

Every non-singular $2\times 2$ matrix $A$ with integer coefficients induces an area preserving endomorphism on $\TT$ which we also denote by $A$ as there is no possibility of confusion.  Moreover, every toral endomorphism (area preserving or not) is homotopic to precisely one such linearly induced map, determined by the action of $f$ in the fundamental group $\pi_1(\TT)$, when the latter is identified with $\ZZ$ in the obvious way. It is invertible if and only if $|\det(A)|=1$. In general, if $|\det(A)|=d$, the pre-image of every point in $\TT$ has cardinality $d$. The same is true for any endomorphism $f$ homotopic to $A$. We say that $d$ is the degree of $f$ or, equivalently, the number of sheets of $f$.

\begin{definition}
We say that an endomorphism $f: \TT \to \TT$ is transitive if there exists $x \in \TT$ whose forward orbit $\{f^n(x): n \geq 0 \}$ is dense in $\TT$. 
\end{definition}

\begin{theorem}\label{main}
Let $f: \TT \to \TT$ be an area preserving toral endomorphism of degree at least two. If $f$ is not transitive, then it is homotopic to a linear map which has a real eigenvalue of modulus one.
\end{theorem}

We remark that a $2 \times 2$ matrix $A$ with integer coefficients satisfying $| \det A | \geq 2$ cannot have non-real eigenvalues on the unit circle. More precisely, one of the following holds:
\begin{enumerate}
\item $A$ has two integer eigenvalues $\lambda, \mu$ with $1 = | \lambda| < |\mu|$
\item $A$ has two integer eigenvalues $\lambda, \mu$ with $1<|\lambda|\leq |\mu|$
\item $A$ has two real irrational eigenvalues $\lambda, \mu$ with $0<|\lambda|<1<|\mu|$
\item $A$ has two real irrational eigenvalues $\lambda, \mu$ with $1<|\lambda|<|\mu|$
\item $A$ has two non-real eigenvalues $\lambda, \overline{\lambda}$ with $1<|\lambda|=|\overline{\lambda}|$.
\end{enumerate}
Usually all of these cases except (1) are referred to as hyperbolic, but  cases (2), (4) and (5) are often called expanding, in which case the word hyperbolic is reserved for the case (3) only. More details on linear toral endomorphisms can be found in Chapter 1 of \cite{aoki1994topological}.

If a non-singular linear map $A:\RR \to \RR$ that satisfies $A (\ZZ) \subset \ZZ$ has a real eigenvalue of modulus one, then it has an invariant one dimensional subspace of rational inclination corresponding to its other eigenvalue. This subspace projects to an invariant circle on $\TT$. This circle can be thickened to create an invariant stripe. In other words, the map induced by $A$ on $\TT$ is not transitive. Conversely, a linear non-singular map on $\RR$ preserving $\ZZ$ which has no  real eigenvalue of modulus one induces a transitive (in fact ergodic --- see \cite{aoki1994topological} or the appendix of \cite{MR1846198}) map on $\TT$ according to Theorem \ref{main}.  We can therefore rephrase Theorem \ref{main} as follows. 

\begin{theorem2}
If a linear map on $\TT$ of degree at least two is transitive, then its whole homotopy class of area preserving endomorphisms consists entirely of transitive elements.
\end{theorem2}

It was shown by Hetzel et al. \cite{MR2321252} that if one chooses a $2 \times 2$ matrix with integer entries at random with respect to the uniform point distribution on some interval $[-N, N] \subset \Z$, then the probability that the matrix has integer eigenvalues tends to zero as $N$ tends to infinity. That provides us with a precise way in which Theorem \ref{main} can be interpreted to say that \emph{most homotopic classes of conservative endomorphisms consist entirely of transitive elements}.



Before turning to the proof of Theorem \ref{main}, we make some remarks regarding notation. The symbol $f$ will always represent an endomorphism of $\TT$ and $A$ will always denote the linear map to which it is homotopic. The canonical projection from $\RR$ to $\TT$ is denoted by $\pi$. A deck transformation is a translation in $\RR$ by an element of $\ZZ$. The symbol $\tf$ will always refer to a lift of $f$ to $\RR$, i.e. a homeomorphism of $\RR$ satisfying $\pi \tf = f \pi$. It satisfies $\tf(x + v) = \tf(x) + Av$ for every $x \in \RR$ and $v \in \ZZ$.

\section{Invariant regular open sets}
An open subset of $\TT$ is called \emph{regular} if it is equal to the interior of its closure. Given $A \subset \TT$ we write $A^\perp =  \TT \setminus \overline{A}$. Note that  an open set $U \subset X$ is regular if and only if $U^{\perp \perp} = U$. 
For any open set $U \subset \TT$ we have inclusion $U \subset U^{\perp \perp}$ and the equality $U^{\perp} = U^{\perp \perp \perp}$. In particular, $U^\perp$ is a regular open set. See \cite{MR2466574} for more details. Note also that if $f: \TT \to \TT$ is a local homeomorphism, then 
\begin{equation}\label{perp and inverse commute}
f^{-1}(A^\perp) = (f^{-1}(A))^\perp
\end{equation} 
for every $A \subset \TT$. We say that two non-empty regular open sets $U,V \subset \TT$ form a complementary pair if $U = V^{\perp}$ (or, equivalently, if $V = U^{\perp}$).

It is well known that for a general continuous map $T$ on a topological space $X$ with countable basis, transitivity is equivalent to say that, given non-empty open sets $U,V \subset X$, there exists some $n\geq 0$ such that $T^{-n}(U) \cap V \neq \emptyset$. Therefore, if $T$ is not transitive, there is a non-dense open set $W \subset X$ such that $T^{-1}(W) \subset W$ (for example, the union of all pre-images of $U$). For general continuous maps, or even for local homeomorphisms, it may not be possible to choose $W$ so that this inclusion becomes an equality.

\begin{definition}
We say that a set $ U \subset \TT$ is strictly invariant under $f: \TT \to \TT$ if $f^{-1}(U) = U$.
\end{definition}

It is the existence of strictly invariant non-trivial open sets in the area preserving setting, and what topological restrictions that the presence of such sets imposes on the map itself, that makes up the core of our approach to the proof of Theorem \ref{main}.

\begin{proposition}\label{nontransitive}
Let $f: \TT \to \TT$ be an area preserving endomorphism. Then the following are equivalent:
\begin{enumerate}
\item $f$ is not transitive,
\item there exists a complementary pair of strictly $f$-invariant regular open sets $U, V  \subset \TT$.
\end{enumerate}
\end{proposition}

\begin{proof}
The implication $(2) \implies (1)$ is trivial. 
To see why $(1)$ implies $(2)$, suppose that $T$ is not transitive. 
Then there are open sets $U_0, V_0$ such that $f^{-n}(U_0) \cap V_0 = \emptyset$ for every $n \geq 0$. 
Let $U_1 = \bigcup_{n \geq 0} f^{-n} (U_0)$. Then $f^{-1}U_1 \subset U_1$. 
Since $\lambda(f^{-1}(U_1)) = \lambda(U_1)$ we deduce that $f^{-1}(U_1) $ is in fact an open and dense subset of $U_1$. 
In particular $\overline{U_1} = \overline{f^{-1} (U_1)} = f^{-1}( \overline{U_1})$, the last equality being true because $f$ is a local homeomorphism. 
Take $V = U_1^\perp$. Then $V$ is a regular open strictly $f$-invariant set. Since $V$ is not dense, $U = V^\perp $ is also a non-empty regular open  strictly invariant set. By construction, $U$ and $V$ form a complementary pair.
\end{proof}

Most of the time, it is not relevant that the sets whose existence is guaranteed by Proposition \ref{nontransitive} are regular, but we need it when describing their fundamental group (Lemma \ref{trivial FG} and Lemma \ref{trivial kernel}). 

Note that Proposition \ref{nontransitive} relies only on the endomorphism being conservative. The existence of non-trivial strictly invariant regular open sets is independent of whether the endomorphism is invertible or not. In the invertible case, there is no restriction on the topological nature of such sets. Indeed, every regular open set is invariant under the identity (which is an area preserving endomorphism). The situation changes drastically in the non-invertible case. Here only sets that ``wrap around" the torus are possible. We now formalize this idea.

\begin{definition}
We say that a connected open set $U \subset \TT$ is nonessential if it is homeomorphic to each  connected component of $\pi^{-1}(U)$ in $\RR$. We say that $U$ is essential if it is not nonessential.
\end{definition}

Some of the important differences between  essential and nonessential sets are captured by the following characterization, the proof of which is straightforward.

\begin{proposition}\label{characterization}
Let $U \subset \TT$ be a connected open subset of $\TT$. Then the following are equivalent:
\begin{enumerate}
\item $U\subset \TT$ is essential.
\item There is a non-trivial deck transformation $T: \RR \to \RR$ such that every connected component of $\pi^{-1}(U)$ is $T$-invariant.
\item If $i:U \to \TT$ is the inclusion, then $i_\star \pi_1(U, x_0)$ is a non-trivial subgroup of $\pi_1(\TT, x_0)$ for some (hence every) $x_0 \in U$.

\end{enumerate}
\end{proposition}

\begin{lemma}\label{e2e and w2w}
Let $f: \TT \to \TT$ be an endomorphism and $U \subset \TT$  a strictly $f$-invariant open set. Then $f$ maps nonessential connected components of $U$ to nonessential connected components of $U$, and it maps essential connected components of $U$ to essential connected components of $U$.
\end{lemma}

\begin{proof}
Let $U \subset \TT$ be open and strictly $f$-invariant. Denote by $\tU$ the pre-image of $U$ by $\pi$ and let $\tf$ be a lift of $f$ to $\RR$. Then $\tU$ is a strictly $\tf$-invariant open set in $\RR$. Therefore, $\tf$ maps connected components of $\tU$ to connected components of $\tU$. If $U_e$ is an essential connected component of $U$, then there is some connected component $\tU_e$ of $\tU$ such that $\pi(\tU_e) = U_e$ and distinct $\tp,\tq \in \tU_e$ such that $\pi(\tp) = \pi(\tq)$. Now, $\tf(\tU_e)$ is a connected component of $\tU$ with $\pi( \tf(\tU_e)) = f(U_e)$. Moreover, $\tf(\tp)$ and $\tf(\tq)$ are distinct points in $\tf(\tU_w)$, both of which are mapped by $\pi$ to the point $f(\pi(p)) = f(\pi(q))$. Hence $\pi: \tf(\tU_e) \to f(U_e)$ is not injective, i.e. $f(U_e)$ is an essential connected component of $U$. 

Now suppose $U_{ne}$ is a nonessential connected component of $U$ and let $\tU_{ne}$ be any connected component of $\pi^{-1}(U_{ne})$. Then $(\tU_{ne} + v)\cap \tU_{ne} = \emptyset $ for every non-zero $v \in \ZZ $. Therefore, $(\tf(\tU_{ne})+ Av) \cap \tf(\tU_{ne}) = \emptyset$ for every non-zero $v \in \ZZ$. Let us write $\tV = \tf(\tU_{ne})$. Suppose that $\tV$ is essential. Then there is some non-zero $w \in \ZZ$ such that $\tV + w = \tV$. It follows by induction that $\tV+ nw = \tV$ for every integer $n$. But we know that the equation $Av = w$ can always be solved for some $v \in \mathbb{Q}^2$. That is equivalent to say that $Av = nw$ can be solved for integer $n$ and $v \in \ZZ$. So $\tV + Av = \tV$ for some non-zero $v \in \ZZ$, a contradiction.
\end{proof}

\begin{lemma}\label{not nonessential}
Let $f: \TT \to \TT$ be an area preserving non-invertible endomorphism and suppose that $U \subset \TT$ is open  and strictly $f$-invariant. Then every connected component of $U$ is essential.
\end{lemma}

\begin{proof}
Let $U \subset \TT$ be a strictly $f$-invariant open set. We denote by $\cU_{ne}$ the set of connected components of $U$ which are nonessential.

 Our aim is to prove that $\mathcal{U}_{ne} = \emptyset$. 
Suppose, for the purpose of contradiction, that $\mathcal{U}_{ne}$ is non-empty. 
Since $\lambda(U_\iota) \leq 1$ for every $U_\iota \in \cU_{ne}$, and since, given any $\epsilon>0$, we have $\lambda(U_\iota)< \epsilon$ for all but finitely many $U_\iota \in \mathcal{U}_{ne}$, 
there is some $U_{m} \in \cU_{ne}$ such that $\lambda(U_{m}) \geq \lambda(U_{\iota})$ for every $U_{\iota} \in \cU_{ne}$. 
Lemma \ref{e2e and w2w} tells us that $f(U_{m}) \in \mathcal{U}_{ne}$. If we can prove that $\lambda(f(U_{m})) > \lambda(U_{m})$, then we have reached a contradiction.

To see why $\lambda(f(U_{m})) > \lambda(U_{m})$, take $\tU_{m} \subset \RR$ such that $\pi:\tU_{m} \to U_{m}$ is injective. 
Let $w \in \ZZ \setminus A(\ZZ)$. 
Such  $w$ exists because $|\det (A)| \geq 2$. Denote by $U_{m}'$ the set $\pi(\tf^{-1}(\tU_{m} + w))$. 
Then $U_{m}  \cap U_{m}' = \emptyset$ but $f(U_{m} ') = f(U_{m})$. 
In particular, $U_{m} \cup U_{m}' \subset f^{-1}(f(U_{m}))$, so $\lambda(U_{m}) \leq \lambda(f(U_{m})) - \lambda(U_{m} ') < \lambda(f(U_{m}))$.
\end{proof}

\begin{lemma}\label{trivial FG}
Let $f: \TT \to \TT$ be an area preserving non-invertible endomorphism and suppose that  $U\subset \TT$ is a regular open strictly $f$-invariant set. Then every connected component of $\pi^{-1}(U)$ is simply connected.
\end{lemma}

\begin{proof}
The case $U = \TT $ is trivial, so in what follows we suppose that $U \neq \TT$. Then $V = U^{\perp}$ is also a regular open set. Moreover, it follows from the identity (\ref{perp and inverse commute}) that $V$ is strictly $f$-invariant. We know from Lemma \ref{not nonessential} that every connected component of $U$ is essential. So is every connected component of $V$. Let $\tU = \pi^{-1}(U)$ and $\tV=\pi^{-1}(V)$. Note that since $U$ and $V$ are regular, so are $\tU$ and $\tV$. It is straightforward to check that connected components of regular open sets are regular. Hence every connected component of $\tU$ and  of $\tV$ is regular.

Let $\tU_0$ be any connected component of $\tU$. We must prove that $\tU$ is simply connected. Suppose, for the purpose of contradiction, that $\tU_0$ is not simply connected. The there exists a loop $\tgamma: S^1 \to \tU_0$ and some $\tx \in \RR \setminus \tU_0$ such that the winding number of $\tgamma$ around $\tx$ is non-zero. 
Since $\tU_0$ is regular, any neighborhood of $\tx$ intersects $\tV$. 
In particular, we can choose $\ty \in \tV$ such that the winding number of 
$\tgamma$ around $\ty$ is non-zero. We denote by $\tV_0$ the connected component of $\tV$ that contains $\ty$. It follows from Lemma \ref{not nonessential} that $\tV_0$ is essential. Hence, by Proposition \ref{characterization}, there is some non-zero $v \in \ZZ$ such that $\ty + n v \in \tV_0$ for every $n \in \Z$. 
Choose $n$ large enough so that the winding number of $\tgamma$ around $\tz = \ty + vn$ is zero. 
Since $\tV_0$ is connected, there is a path $\tsigma$ from $\ty$ to $\tz$. 
Since the winding number of $\tgamma$ is different at $\ty$ and $\tz$, we conclude that $\image \tgamma \cap \image \tsigma \neq \emptyset$. 
But that is absurd because $\image \tgamma \subset \tU_0$ and $\image \tsigma \subset \tV_0$. 
\end{proof}

\begin{lemma}\label{trivial kernel}
Let $f:\TT \to \TT $ be a non-invertible area preserving endomorphism and suppose that $U\subset \TT$ is a regular open strictly $f$-invariant set. If $U_0$ is a connected component of $U$ and $i:U_0 \to \TT$ the inclusion map, then the kernel of $i_\star: \pi_1(U_0) \to \pi_1(\TT)$ is trivial.
\end{lemma}

\begin{proof}
This is a direct consequence of Lemma \ref{trivial FG}.
\end{proof}

It is instructive to consider the example $f(x,y) = (2x, y + \frac{1}{2}) \ \mod \ZZ$ on $\TT$. The set $U = \T \times ((0,\frac{1}{4}) \cup (\frac{1}{2},\frac{3}{4}))$ is a regular open strictly invariant set under $f$. In accordance with Lemma \ref{trivial FG}, every connected component of $\pi^{-1}(U)$ is simply connected. Of course, $U$ itself is not connected. Neither is any of its two connected components invariant under $f$. However, each of them is invariant under $f^2$. The next Lemma shows that this is to be expected.

\begin{lemma}\label{invariant under iterate}
Let $f: \TT \to \TT $ be an area preserving non-invertible endomorphism and suppose that $U \subset \TT$ is a strictly $f$-invariant regular open set. If $U_0$ is a connected component of $U$ then there exists $n\geq 1$ such that $U_0$ is strictly invariant under $f^n$.
\end{lemma}

\begin{proof}
Let $f$ and $U$ be as in the hypothesis of the lemma and suppose that $U_0$ is a connected component of $U$. Since $f$ is area preserving, there exists $n \geq 1$ such that $f^n(U_0) \cap U_0 \neq \emptyset$. But then $f^n(U_0) = U_0$ according to Lemma \ref{e2e and w2w}. Hence $U_0 \subset f^{-n}(f^n(U_0)) = f^{-n}(U_0)$. Since $f$ is area preserving, it follows that  $U_0$ has full $\lambda$-measure in $f^{-n}(U_0)$. In particular, $U_0$ is  dense in $f^{-n}(U_0)$. But both $U_0$ and $f^{-n}(U_0)$ are regular. They must therefore coincide.
\end{proof}

\section{Some further auxiliaries}

The proof of Theorem \ref{main} is centered around the following result from the basic theory of covering spaces.

\begin{theorem}\label{number of sheets}
Let $X$ be a path connected topological space. Suppose that  $(\tilde{X},p)$ is a covering space of $X$ and that $p(\tx_0)=x_0$. Then the number of sheets of $p$ is equal to the index of $p_\star \pi_1(\tilde{X}, \tx_0)$ in $\pi_1(X,x_0)$.
\end{theorem}

We shall also make use of the following  intersection property.

\begin{lemma}\label{intersecting loops}
Suppose that $\gamma$ and $\sigma$ are loops in $\TT$ such that $[\sigma]$ and $[\gamma]$ are linearly independent in $\ZZ$. Then  $\gamma$ and $\sigma$ intersect.
\end{lemma}

Rigorous proof Lemma \ref{intersecting loops} is more subtle than a cursory glance may suggest. The author found no suitable reference, so for the sake of completeness we include a proof based on elementary homotopy theory. 

\begin{proof}[Proof of Lemma \ref{intersecting loops}]
Let $\Gamma: \R \to \TT$ and $\Sigma: \R \to \TT$ be periodic extensions of $\gamma$ and $\sigma$ respectively. That is, $\Gamma(t+n) = \gamma(t)$ and $\Sigma(t+n) = \sigma(t)$ for every $t \in [0,1]$ and every $n \in \Z$. Let $\tilde{\Gamma}$ and $\tilde{\Sigma}$ be lifts of $\Gamma$ and $\Sigma$ to $\RR$.  Up to a change of coordinates, we can (and do) assume that $\tG(1)-\tG(0) = (1,0)$ and $\tS(1)-\tS(0) = (0,1)$. Then $\tG(t+n) = \tG(t) + (n,0)$ and $\tS(t+n) = \tS(t)+(0,n)$ for every $t \in[0,1]$ and every $n \in \Z$. In particular there exists $K$ such that $\| \tG(t)-(t,0)\| \leq K$ and $\| \tS(t) - (0,t)\| \leq K$ for every $t \in \R$.

Let $S^1 = \{(s,t) \in \RR: \|(s,t)\| = 1\}$ and let 
\begin{equation}
R = \{ r \geq 0 : \tG(rs) \neq \tS(rt) \text{ for every }(s,t) \in S^1 \}.
\end{equation}
For every $r \in R$ we can define a map $h_r: S^1 \to S^1$ by
\begin{equation}
h_r(x,y) = \frac{\tilde{\Gamma}(rs) - \tilde{\Sigma}(rt)}{\| \tilde{\Gamma}(rs) - \tilde{\Sigma}(rt)\|}.
\end{equation}
  It is straightforward to check that the family $h_r$ depends continuously on $r$ in $C^0(S^1, S^1)$. We shall  prove that $R$ contains  $(2K, \infty)$ and that as $r$ tends to infinity, $h_r$ converges uniformly to the involution $I: S^1 \to S^1$, $(s,t) \mapsto (s,-t)$. The proof may then be concluded as follows: Suppose that $\gamma$ and $\sigma$ do not intersect. Then $R = [0,\infty)$. The uniform convergence of $h_r$ to $I$ as $r \to \infty$ implies that $h_0$ is homotopic to the involution $I$. But this is abusurd, since $h_0$ is a constant map.  Therefore, $R \neq [0,\infty)$, so $\tG(rs) = \tS(rt)$ for some $r \geq 0$ and some $(s,t)\in S^1$. Therefore (the images of ) $\gamma$ and $\sigma$ must intersect.

It remains to prove that $h_r$ converges uniformly to $I$ as $r \to \infty$.
For notational convenience, write $u = \tG(rs) - \tS(rt)$ and $v = I(s,t) = (s,-t)$ (leaving the dependence on $r, s$ and $t$ understood implicitly). Thus $h_r = u / \| u \|$ and we must prove that 
\begin{equation}
\left\| \frac{u}{\| u \|} - v \right\| 
\end{equation}
converges to zero uniformly on $S^1$ as $r \to \infty$.

Note that
\begin{equation}
\|u-rv \| \leq \|\tG(rs)-(rs,0)\| + \|\tS(rt)-(0, rt)\| \leq 2K
\end{equation}
for every $s,t \in \R$ and every $r \geq 0$.
Note also that $\| v \| = 1$ so, using the triangle inequality, we get 
\begin{equation}
r-2K \leq \| u \| \leq r+2K,
\end{equation}
whence it follows that $(2K, \infty) \supset R$. Moreover, for  $r>2K$ we have
\begin{equation}
\left| \frac{1}{\|u\|} - \frac{1}{r} \right| \leq \frac{2K}{r(r-2K)}.
\end{equation}
Hence 
\begin{align}
 \| h_r(s,t)- I(s,t)\| &= \left\| \frac{u}{\|u\| } - v \right\| \\ & \leq \left\| \frac{u}{\|u\|} - \frac{u}{r} \right\| + \left\| \frac{u}{r}- v \right\| \\ & \leq  \frac{(r+2K)2K}{r(r-rK)} + \frac{2K}{r},
\end{align}
which converges to $0$ as $r \to \infty$. This completes the proof.

\end{proof}

\section{Proof of theorem \ref{main}}
Let $f: \TT \to \TT$ be a non-invertible conservative endomorphism which  is not transitive. Then we know from Proposition \ref{nontransitive} that there exists a complementary pair of  strictly $f$-invariant regular open sets $U,V \subset \TT$. Lemma \ref{not nonessential} tells us that all connected components of $U$ and $V$ are essential. Hence there exist a loop $\gamma$ in $U$ and a loop $\sigma$ in $V$, neither of which is homotopic in $\TT$ to a constant path. In other words, $[\gamma]$ and $[\sigma]$ are non-trivial elements in $\pi_1(U)$ and $\pi_1(V)$, respectively, as well as in $\pi_1(\TT)$.

Note that $[\gamma]$ and $[\sigma]$ must be collinear elements of $\pi_1(\TT)$, or else they would intersect according to Lemma \ref{intersecting loops}. Note also that $[\gamma]$ and $[\sigma]$  are eigenvectors of the action $A:\ZZ \to \ZZ$, (when seen as elements of $\ZZ$). Were it not so, $A[\gamma]$ and $[\sigma]$ would not be collinear, implying $f(U) \cap V \neq \emptyset$, contradicting the invariance of $U$. 

We are now able to claim that the eigenvalues of $A$ are non-zero integers. Indeed, if they were irrational or non-real, there would be no eigenvector in $\ZZ$.
But if an integer matrix has rational eigenvalues, these must indeed be integers. This follows from the rational roots theorem and the fact that the characteristic equation is a monic polynomial. Finally, $A$ is non-singular so it cannot have a vanishing eigenvalue.

 It remains to prove that one of the eigenvalues of $A$ is equal to $\pm 1$. To this end, let $k,\ell \in \Z$ be eigenvalues of $A$ and suppose without loss of generality that $k$ is the eigenvalue associated to $[\gamma]$. Note that the determinant of $A$ is equal to $k \ell$. In particular, the number of sheets of $f$ is equal to $|k \ell|$. If $|k|=1$ there is nothing to prove. So, to prove the theorem we suppose that $|k| \geq 2$ and shall deduce that in this case $|\ell| = 1$. 

Let $U_0$ be the connected component of $U$ containing $\gamma$, and $i: U_0 \to \TT$ the inclusion map. We have already observed that every element $[\tau]$ of $i_\star \pi_1(U_0)$ must be collinear with $[\gamma]$. It follows that, with the canonical identification of $\pi_1(\TT)$ with $\ZZ$, 
the subgroup $ i_\star \pi_1(U)$ in $\ZZ$ is of the form $\{rv: r \in \Z \}$ for some non-zero $v \in \ZZ$.  By Lemma \ref{trivial kernel}, $i_\star$ is a monomorphism. Hence $\pi_1(U_0)=\{ rv: r \in \Z \} \subset \ZZ$.

By Lemma \ref{invariant under iterate} there exists $n \geq 1$ such that $U_0$ is strictly invariant under $f^n$. Notice that the pair $(U_0, f^n \vert U_0)$ can be seen as a covering space of $U_0$. 
Therefore, Theorem \ref{number of sheets} implies that the number of sheets of $f^n\vert U_0$ (and hence of $f^n$) is equal to the index of  $f_\star^n (\pi_1(U_0))$ as a subgroup of $\pi_1(U_0)$. Recall that every $[\gamma] \in \pi_1(U_0)$ is an eigenvector of $A^n$ with eigenvalue $k^n$.  Hence $f_\star^n (\pi_1(U_0)) = \{r k^n v: r \in \Z \}$ so that the index of $f_\star^n (\pi_1(U_0))$ in $\pi_1(U_0)$ is $|k|^n$. On the other hand, we have already observed that the number of sheets of $f$ (hence of $f\vert U_0$) is equal to $|k\ell|$, so the number of sheets of $f^n \vert U_0$ must be equal to $|k\ell|^n$. It follows that $|\ell| = 1$ as required.

\bibliographystyle{plain}

\bibliography{transitive}

\def\cprime{$'$}
\begin{thebibliography}{1}

\bibitem{2015arXiv150306501A}
M.~{Andersson} and S.~{Gan}.
\newblock {Transitivity of conservative diffeomorphisms isotopic to Anosov on
  $\mathbb{T}^3$}.
\newblock {\em arXiv:1503.06501}, March 2015.

\bibitem{aoki1994topological}
N.~Aoki and K.~Hiraide.
\newblock {\em Topological Theory of Dynamical Systems: Recent Advances}.
\newblock North-Holland Mathematical Library. Elsevier Science, 1994.

\bibitem{MR1846198}
F.~Brini and S.~Siboni.
\newblock Estimates of correlation decay in auto/endomorphisms of the
  {$n$}-torus.
\newblock {\em Comput. Math. Appl.}, 42(6-7):941--951, 2001.

\bibitem{MR2466574}
Steven Givant and Paul Halmos.
\newblock {\em Introduction to {B}oolean algebras}.
\newblock Undergraduate Texts in Mathematics. Springer, New York, 2009.

\bibitem{MR3043026}
Baolin He and Shaobo Gan.
\newblock Robustly non-hyperbolic transitive endomorphisms on {$\Bbb{T}^2$}.
\newblock {\em Proc. Amer. Math. Soc.}, 141(7):2453--2465, 2013.

\bibitem{MR2321252}
Andrew~J. Hetzel, Jay~S. Liew, and Kent~E. Morrison.
\newblock The probability that a matrix of integers is diagonalizable.
\newblock {\em Amer. Math. Monthly}, 114(6):491--499, 2007.

\bibitem{LPV}
Cristina Lizana, Vilton Pinheiro, and Paulo Varandas.
\newblock Contribution to the ergodic theory of robustly transitive maps.
\newblock preprint.

\bibitem{MR3082540}
Cristina Lizana and Enrique Pujals.
\newblock Robust transitivity for endomorphisms.
\newblock {\em Ergodic Theory Dynam. Systems}, 33(4):1082--1114, 2013.

\bibitem{MR1469436}
Naoya Sumi.
\newblock A class of differentiable toral maps which are topologically mixing.
\newblock {\em Proc. Amer. Math. Soc.}, 127(3):915--924, 1999.

\end{thebibliography}

\end{document}